\documentclass{amsart}
\usepackage{amssymb,latexsym, mathrsfs,amsthm,amsmath}
\usepackage{mathabx}
\usepackage{breqn}
\newtheorem{example}{Example}[section]

\newtheorem{lemma}{Lemma}[section]

\newtheorem{defn}{Definition}[section]

\newtheorem{cor}{Corollary}[section]

\numberwithin{equation}{section}

\newcommand\restr[2]{{
  \left.\kern-\nulldelimiterspace 
  #1 
  \vphantom{\big|} 
  \right|_{#2} 
  }}

\newcommand{\I}{\mathcal{U}}
\newcommand{\Ov}{\mathcal{O}}

\newcommand{\Zz}{\mathbb{Z}}
\newcommand{\N}{\mathbb{N}}

\DeclareMathOperator{\lcm}{\textup{lcm}}
\DeclareMathOperator{\ctype}{\textup{ctype}}
\DeclareMathOperator{\aut}{\textup{Aut}}
\DeclareMathOperator{\sym}{\textup{Sym}}
\DeclareMathOperator{\ord}{\textup{ord}}

\begin{document}

\title{The cycle index of the automorphism group of $\mathbb{Z}_n$}

\author{Vladimir Bo\v{z}ovi\'c and \v{Z}ana Kovijani\'c Vuki\'cevi\'c}

\begin{abstract}
We consider the group action of the automorphism group $\I_n=\aut(\Zz_n)$ on the set $\Zz_n$, that is the set of residue classes modulo $n$. Clearly, this group action provides a representation of $\I_n$ as a permutation group acting on $n$ points. One problem to be solved regarding this group action is to find its cycle index. Once it is found, there appears a vast class of related enumerative and computational problems with interesting applications. We provided the cycle index of specified group action in two ways. One of them is more abstract and hence compact, while another one is basically procedure of composing the cycle index from some \textit{building blocks}. However, those \textit{building blocks} are also well explained and finally presented in very detailed fashion.  
\end{abstract}

\maketitle

\section{Introduction}
Let $\I_n=\aut(\Zz_n)$ be the automorphism group of the cyclic additive group of residues modulo $n$. Throughout this paper, we treat $\Zz_n$  interchangeably, merely as a set $\{0,1,\ldots,n-1\}$ or as the additive, cyclic group. However, the context in which it is used will clearly determine its meaning. 

As it is well known, $\I_n$ is isomorphic to the multiplicative group of those integers in $\mathbb{Z}_n$ that are relatively prime to $n$, i.e. 
\[\I_n=\{\pi_a: \mathbb{Z}_n \rightarrow \mathbb{Z}_n \mid \pi_a(x)=ax \pmod{n},\,1 \leq a \leq n,\, (a,n)=1\}.\]
Based on that isomorphism, a mapping $\pi_a \in \I_n$, can be identified with an element $a \in \mathbb{Z}_n$, $(a,n)=1$. Further, we will be frequently using that convenient isomorphic correspondence, without risk of misconception. Hence, the natural group action of the group $\I_n$ on the set of elements of $\Zz_n$ could be seen as 
 \[(x,a)\rightarrow ax \pmod{n} \;\; \;\;(a \in \I_n,\;x\in \Zz_n),\]
Clearly, the automorphism group, $\I_n$, represented in the described way is a permutation group acting on $n$ points. Although elementary in its nature, it was surprising fact that, to the best of our knowledge, the cycle index of the described group action is still missing. Besides PhD thesis, \cite{Boz-thesis}, in which the problem is treated to some extent, there is only one paper, \cite{Wei}, that partially deals with the similar group action, and yet substantially different. Therefore, finding the cycle index, $\mathcal{Z}_{(\I_n,\mathbb{Z}_n)}$, of the described group action emerges as the main goal of this paper.  

Based on somewhat distinct, although complement approaches, we get two forms of the same result. The first one that looks more general, whereas another one is technically more detailed, supplying raw structure of cycle index $\mathcal{Z}_{(\I_n,\mathbb{Z}_n)}$. In the first approach, we get nice, compact result from the Corollary \ref{nice-compact}
\[\mathcal{Z}_{\left(\I_n,\mathbb{Z}_n\right)}= \frac{1}{\phi(n)}\sum_{a \in  \I_n}\prod_{d \mid n} x_{r_a(d)}^{\frac{\phi(d)}{r_a(d)}},\]
where $r_a(d)$ is a multiplicative order of an integer $a \in \I_n$ modulo $d$. 

In the second approach, we use the fact of the direct decomposition of additive abelian group and corresponding decomposition of its automorphism group 
\[\Zz_n \cong \bigoplus_{i=1}^s\Zz_{p_i^{\alpha_i}},\,\, \I_n\cong\bigoplus_{i=1}^s\I_{p_i^{\alpha_i}}.\]
As $n$ can be decomposed as a product of prime number powers, then we proceed by finding cycle index $\mathcal{Z}_{(\I_{p^{\alpha}},\mathbb{Z}_{p^{\alpha}})}$, for a prime number $p$ and $\alpha \in \mathbb{N}$, as a groundwork for utilization of known result given in \cite{HarHigh}, that is an algorithm for finding cycle index of direct product of permutation groups.

Once the cycle index is found, there is a vast class of enumerative and combinatorial problems that could be related to it. For example, one of the classical enumerative, combinatorial "targets" is a number of orbits or equivalence classes of subsets of $\Zz_n$.

\section{Preliminaries}
In this section, we bring up some basic, auxiliary results, notation and assumptions that will be used for the rest of the paper.

\begin{itemize}
\item By \textit{natural number} we assume positive integer.
\item By $\displaystyle \biguplus$ we denote disjoint union of sets.
\item The label $\phi$ will be exclusively used for the Euler's phi function.
\item By $C_n$ we denote a cyclic group of $n$ elements.
\item By $\sym(M)$ we denote the full symmetric group on the set $M$.
\item By $(a,b)$ we denote $\gcd(a,b)$ and by $[a,b]$ $\lcm(a,b)$.
\item By $\ord(g)$ we denote the order of an element $g$ in a group $G$.
\end{itemize}

Regarding a topic of group action, we slightly changed definitions of classical notions and accordingly, introduce some new notation. 
\begin{defn} \label{partial-cycle-index-defn}(Type of a Permutation) Let $P$ be a set with $|P| = n$. A permutation $\pi \in \sym(P)$
is of the type $(\lambda_1, \lambda_2, \ldots ,\lambda_n)$, iff $\pi$ can be written as the composition of $\lambda_i$ disjointed cycles of length
$i$, for $i = 1, \ldots ,n$. Hence, by $\lambda_i(\pi)$ we mean the number of cycles of length $i$ in the decomposition of $\pi$ into disjoint cycles. Shortly, we write
\[\ctype(\pi)=\prod_{i=1}^nx_i^{\lambda_i(\pi)}.\]
\end{defn}

Note that variable $x_i$ has only formal meaning, referring to the cycle of the length $i$.

\begin{defn}(Partial Cycle Index) Let $P$ be a set of $|P| = n$ elements and let $\Gamma$ be a subset of finite permutation group $\mathcal{G}_P$ acting on $P$.
\textit{The partial cycle index} of a subset $\Gamma \subseteq \mathcal{G}_P$ is defined as a polynomial in $n$ indeterminates
$x_1, \ldots , x_n$, defined as:
\[\mathcal{Z}_{(\Gamma,P)}(\mathcal{G}_P):= \frac{1}{|\mathcal{G}_P|}\sum_{\pi \in \Gamma}\ctype(\pi)=\frac{1}{|\mathcal{G}_P|}\sum_{\pi \in \Gamma} \prod_{i=1}^nx_i^{\lambda_i(\pi)}.\]
When $\Gamma = \mathcal{G}_P$, then $\mathcal{Z}_{(\mathcal{G}_P,P)}(\mathcal{G}_P)$ is called \textit{the cycle index} of $\mathcal{G}_P$ on $P$, or shortly $\mathcal{Z}_{(\mathcal{G}_P,P)}$.

\end{defn}
It should be emphasized that a natural number $n$, in the previous definition, is upper limit regarding the number of indeterminates of cycle index polynomial. The actuall number of indeterminates that could appear in the cycle index polynomial is $\lambda(n)$, that is the maximal order among of all orders of elements in $\mathcal{G}_P$.

Directly from the previous definition, we have this simple observation.
\begin{lemma}
Let $\mathcal{G}_P$ be a permutation group acting on a set $P$ of  $n$ elements and a $\Gamma_i \subseteq \mathcal{G}_P$, $i=1,2, \ldots, k$, where $k \in \mathbb{N}$, be a set partition of the set of elements of $\mathcal{G}_P$. Then, the cycle index of $\mathcal{G}_P$ on $P$ is
\[\mathcal{Z}_{(\mathcal{G}_P,P)}=\sum_{i=1}^k \mathcal{Z}_{(\Gamma_i,\mathcal{G}_P)}(\mathcal{G}_P).\]

\end{lemma}
\begin{proof}
It is direct consequence of the Definition \ref{partial-cycle-index-defn} and the fact that $\mathcal{G}_P$ is disjoint union of $\Gamma_i$, $i=1,2, \ldots, k$, i.e. 
\[ \mathcal{G}_P=\biguplus_{i=1}^k \Gamma_i.\]
\end{proof}

Note that  $\Gamma$ in the previous definition is a subset, not necessarily a subgroup of $\mathcal{G}_P$. This notion of partial cycle index will be helpful later.  

Let us introduce the notion of $(r,k)-$coprime residue set in $\Zz_n$.
\begin{defn} Let $r,k$ be natural numbers such that $\gcd(r,k)=1$, $r<k$ and let $k$ be a divisor of natural number $n$. The set of integers
\[\I_{n}^{r}(k)=\{x\in \I_n \mid \,x\equiv r \pmod{k}\}\]
is called $(r,k)-$coprime residue set in $\Zz_n$.
\end{defn}
We prove that any $(r,k)-$coprime set in $\Zz_n$ is not empty. It is essentially a restatement and slight modification of the result given in Lemma 2, page 32, \cite{Cohn}. This result will be helpful later, in characterization of orbits of the examined group action.
\begin{lemma} \label{induklema}
Let $r,k,\ell,n$ be natural numbers such that $\gcd(r,k)=1$, $r<k$ and $n=k\ell$. Then $(r,k)$-coprime set $\I_{n}^{r}(k)$ is
nonempty.
\end{lemma}
\begin{proof}
We prove for given $r,k$ and $n$ and $\gcd(r,k)=1$, there exists $t$ such that
\[\gcd(r+ut,n)=1\]
Let $p_i^{v_i}$ be a general prime power divisor of $n$. Then, there exists $t_i$ such that
\[\gcd(r+kt_i,p_i^{v_i})=1\]
Namely, if $p_i \mid k$, then $p_i \nmid r$ and $t_i=0$ suffices. If $p_i \nmid k$, than any number $t_i$ such that
\[t_i \not\equiv -r/k \pmod{p_i}\] will work.
By Chinese Reminder Theorem, there exists $t$ such that
\[t \equiv t_i \pmod {p_i}\]
and $\gcd(r+kt,n)=1$.
We need to prove that there exists $x\in \I_n$ such that $x\equiv r \pmod{k}$. Let $x \equiv r+kt \pmod{n}$. Since $k \mid n$ then
$x \equiv r \pmod {k}$. Also, it is easy to see that $\gcd(x,n)=1$ and therefore $x \in \I_n$.
\end{proof}

The following lemma, almost textbook statement, we introduce without proof.

\begin{lemma}\label{genistcikl}
Let $C_n = \langle a \rangle$ be a cyclic group of $n$ elements and let $d \in \mathbb{N}$ be an integer such that $d \mid n$. By $A_d$ denote the set of all elements of $C_n$ of order $d$. Then, 
\[ A_d=\left\{a^{ \frac{n}{d}t} \mid  t \in \mathbb{N} \text{ and }\gcd(t, d) = 1\right\}.\]
Apparently $| A_d|=\phi(d)$. Also, $C_n$ is disjoint union of those sets, that is
\[C_n=\biguplus_{d \mid n} A_d.\]

\end{lemma}
\qed

Let $\Omega_n^d$, where $d \mid n$, be the set of elements of additive order $d$ in the $\Zz_n$. Then, according to Lemma \ref{genistcikl}, we have
\[\Zz_n=\biguplus_{d \mid n}\Omega_n^d\]
and $|\Omega_n^d|=\phi(d)$.
The partition of $\Zz_n$, we just specified, will play important role in the analysis of the group action we are dealing with.

\section{Cycle index of $\I_n$ - general case}

As pointed in the Introduction, we consider the natural group action of the group $\I_n$ on the set of elements of $\Zz_n$, given by
 \[(x,a)\rightarrow ax \pmod{n}\;\; \;\;(a \in \I_n,\;x\in \Zz_n).\]

Based on results in the previous section, we prove that the typical orbit of the aforementioned group action is actually $\Omega_n^d$, that is the set of elements of additive order $d$ in the $\Zz_n$.
\begin{lemma} \label{transitonzn} Let $d$, $n$ be natural numbers, such that $d \mid n$. Then 
\[\Omega_n^d=\frac{n}{d}\cdot \I_d=\left \lbrace   \frac{n}{d}t \mid 1 \leq t \leq d \text{ and } \gcd(t,d)=1\right\rbrace . \]
Also, $\Omega_n^d$ is an orbit under the action of the group $\I_n$ on $\Zz_n$.
\end{lemma}
\begin{proof}
The first fact of the claim is trivial consequence of Lemma \ref{genistcikl}, so we have
\[\Omega_n^d=\frac{n}{d}\cdot \I_d.\]
We prove that $\Omega_n^d$ is an orbit in the action of  $\I_n$ on $\Zz_n$.
Let $x$ and $y$ be elements of order $d$. Then we have $x=(n/d)k_1$ and $y=(n/d)k_2$ where $k_1,k_2 \in \I_d$. Therefore, there exists $k \in \I_d$ and $k_1=kk_2$. 

On the other hand, Lemma \ref{induklema} claims that $\I_{n}^{k}(d)$ is nonempty, i.e. the existence of an element $h \in \I_n$ such that $h \equiv k \pmod{d}$. Clearly $k_1\equiv hk_2 \pmod{d}$. By multiplying both sides by $(n/d)$ we have
\[x \equiv hy \pmod{n}.\] 
Thus, $\I_n$ is transitive on the set of elements of (additive) order $d$.
\end{proof}

Let $a$, $d$ be natural numbers such that $\gcd(a,d)=1$. Denote by $r_a(d)$ the order of $a$ with respect to modulo $d$, i.e. 
 \[ r_a(d)=\min \{k \in \N \mid a^k \equiv 1 \pmod{d}\}.\]
 
The following lemma has a key role in a description of how the mapping ${\pi_a: \mathbb{Z}_n \rightarrow \mathbb{Z}_n}$, defined as $\pi_a(x)=ax \pmod{n}$, acts on a typical orbit $\Omega_n^d$.
\begin{lemma} \label{gen-na-gen}
Let $d,n,a$ be natural numbers such that $d \mid n$ and $\gcd(a,n)=1$. Consider $\tau =\restr{\pi_a}{\Omega_n^d}$, that is restriction of the mapping $\pi_a: \mathbb{Z}_n \rightarrow \mathbb{Z}_n$, defined as $\pi_a(x)=ax \pmod{n}$, on $\Omega_n^d$. In other words, 
\[\tau(x) = ax \pmod{n}, \text{ where } x \in \Omega_n^d.\]
 
Then, $\tau$ is a permutation of $\Omega_n^d$ and 
\[\ctype(\tau)=x_k^m, \;\; \text{where}\;\;k=r_a(d),\,\, m=\frac{\phi(d)}{k}.\]

\end{lemma}

\begin{proof}
We know, from Lemma \ref{transitonzn}, that $\Omega_n^d$ is an orbit of $\pi_a$, so we conclude that $\tau: \Omega_n^d \rightarrow \Omega_n^d$. Since, $\pi_a$ is a bijection, then $\tau$ is certainly injection on $\Omega_n^d$. However, $\Omega_n^d$ is finite set, so $\tau$ must be bijection.

According to Lemma \ref{transitonzn}, an arbitrary element $c \in \Omega_n^d$ is of the form 
\[c=\frac{n}{d}{v}, \; \text{where}\;\gcd(v,d)=1.\]
Let us consider the cycle that $c$ belong to, considering the mapping $\tau$ and let $s$ be the length of that cycle
\[(c \rightarrow ac \rightarrow  a^2c \rightarrow \ldots  \rightarrow a^{s-1}c).\]
From $a^kv \equiv v \pmod{d}$ it follows 
\[(a^k-1)v=dt,\text{ for some } t \in \mathbb{Z}.\]
Therefore,
\[(a^k-1)\frac{n}{d}v=nt, \text{ that means } a^kc \equiv c \pmod{n}.\]

Since $s$ is the least number such that $a^sc \equiv c \pmod{n}$, then $s \leq k$.

On the other hand, from  $a^sc \equiv c \pmod{n}$, it follows 
\[a^sv \equiv v \pmod{d}. \] 
Since, $\gcd(v,d)=1$, we conclude $ a^s \equiv 1 \pmod{d}$.
However, $k=r_a(d)$ and thus $k \leq s$, so we finally get $k=s$. 
Since $c$ is an arbitrary element in $\Omega_n^d$, we conclude that every cycle of $\tau$ is of the same length $k=r_a(d)$. 
\end{proof}


\begin{cor}\label{nice-compact}
Let $n,a$ be natural numbers such that $d \mid n$ and $\gcd(a,n)=1$. Then, bijection $\pi_a: \mathbb{Z}_n \rightarrow \mathbb{Z}_n$, defined as $\pi_a(x)=ax \pmod{n}$, has cyclic structure
\[\ctype(a)=\prod_{d \mid n} x_{r_a(d)}^{\frac{\phi(d)}{r_a(d)}}.\]

Accordingly, the cycle index of $\I_n$, acting on the set $\mathbb{Z}_n$, is
\begin{equation} \label{cycle-index-compact}
\mathcal{Z}_{(\I_n,\mathbb{Z}_n)}= \frac{1}{\phi(n)}\sum_{a \in  \I_n}\prod_{d \mid n} x_{r_a(d)}^{\frac{\phi(d)}{r_a(d)}}. \tag{$\Asterisk$}
\end{equation}

\end{cor}
\qed

We should notice that the number of indeterminates in the polynomial $\mathcal{Z}_{(\I_n,\mathbb{Z}_n)}$ is actually equal to the maximal multiplicative order among elements in $\I_n$, that is 

\[\lambda(n)=\max\{r_a(n) \mid a \in \I_n\}.\] 
As it is principally analysed and resolved in \cite{Char}, for a natural number $n$ represented as a product of powers of prime numbers 
$n=\prod_{i=1}^k p_i^{e_i}$, we have that $\lambda(n)$ is equal to the least common multiplier of $\lambda(p_i^{e_i})$, for $i=1,\ldots,k$. Thus,
\[\lambda(n)= \left[ \lambda(p_1^{e_1}), \lambda(p_2^{e_2}),\ldots,\lambda(p_k^{e_k})\right],\]
where 
\[
 \lambda(p^e)= 
  \begin{cases}
   \phi(p^e) & \text{if p an odd prime}, e \geq 1\\
   \phi(2^{e}) & \text{if } p=2, e\leq 2\\
   2^{e-2} & \text{if }p=2, e>2.  \\
  \end{cases}
\]
Therefore, only indeterminates $x_1,x_2,\ldots,x_{\lambda(n)}$ appear in the cycle index $\mathcal{Z}_{(\I_n,\mathbb{Z}_n)}$.

\section{Cycle index of $\I_{p^m}$}
In this section, we present more comprehensive look over the cycle index of the group $\I_{p^m}$ acting on $\mathbb{Z}_{p^m}$. Since the algebraic structure of $\I_{p^m}$ differs in two basic cases: when $p=2$ and when $p$ is an odd prime number, we will consider the same two cases in the course of finding cycle index of $\I_n$. 

Once a cycle index of $\I_{p^m}$, for a prime number $p$, is found, it will serve as a building block for compounding cycle index of $\I_n$, where $n$ is naturally represented as a product of powers of prime numbers.

\subsection{Cycle index of $\I_{2^m}$}
We start with the case of $\I_{2^m}$ acting on $\mathbb{Z}_{2^m}$. First of all, we need to determine how elements of $\I_{2^m}$ look like, considering the fact that it is not a cyclic group for $m \geq 3$, but direct product $C_2 \times C_{2^{m-2}}$. Still, there is a way for all elements of $\I_{2^m}$ to be represented in  functional form.

\begin{defn}
An integer $a$ is said to be a semi-primitive root modulo $n$ if the order of a modulo $n$ is equal to $\phi(n)/2$, where $\phi$ is Euler function.
\end{defn}

It is shown in \cite{Gauss} that 3 is a semi-primitive root modulo $2^m$. Thus, the order of $3$ modulo $2^m$ is $2^{m-2}$, for any integer $m \geq 3$ and 
\[\I_{2^m}=\left\{\pm 3^i \pmod{2^m} : i=1,\ldots,2^{m-2}\right\}.\]
Then, the following lemma is just rewording of the previous fact.

\begin{lemma} \label{generator-3} 
 
For an arbitrary element $w \in \I_{2^m}$, if $m \geq 3$, there exists unique pair $(a,b)$, $a \in \{0,1\}$ and $b\in \{0,1,\ldots,2^{m-2}-1\}$ such that
\[w=(-1)^a3^b.\]
\end{lemma}
\qed

It might be useful to mention that 5 is a semi-primitive for $\I_{2^m}$, $m  \geq 3$ as well. 
 
The claims of the following two corollaries are either obvious or directly coming from the Lemma \ref{generator-3}.  
 
\begin{cor} \label{rd-plus3-nam}

Let $a=3^{2^sr}$, $d=2^l$, where $l \geq 1$, $s \geq 0$, $r =1,3,\ldots, 2^{m-2-s}-1$. Then

\[
 r_a(d)= 
  \begin{cases}
   1 & \text{if } l=1;s\geq 0, \\
   2 & \text{if } l=2;s=0, \\
   1 & \text{if } l=2; s \geq 1, \\
   2^{l-2-s} & \text{if } l\geq 3; s < l-2, \\
   1& \text{if } l\geq 3; s\geq  l-2. \\
  \end{cases}
\]

\end{cor}
\qed

\begin{cor} \label{rd-minus3-nam}

Let $a=-3^{2^sr}$, $d=2^l$, where $l \geq 1$, $s \geq 0$, $r =1,3,\ldots, 2^{m-2-s}-1$. Then

\[
 r_a(d)= 
  \begin{cases}
   1 & \text{if } l=1;s\geq 0, \\
   1 & \text{if } l=2;s=0, \\
   2 & \text{if } l=2; s \geq 1, \\
   2^{l-2-s} & \text{if } l\geq 3; s < l-2, \\
   2& \text{if } l\geq 3; s\geq  l-2. \\
  \end{cases}
\]

\end{cor}

\qed

\begin{lemma} \label{gama1-parni}
Let us consider
\[\Gamma_1=\{3^{2^sr} \mid s=0,1,\ldots,m-3,\,r =0,1,3,\ldots, 2^{m-2-s}-1\}\]
as a subset of the group $\I_{2^m}$, for $m\geq 3$. Then, the partial cycle index of this set is
\[\mathcal{Z}_{(\Gamma_1,\Zz_{2^m})}(\I_{2^m})=\frac{1}{2^{m-1}}\left(x_1^{2^m}+ 2^{m-3}x_1^2x_2\prod_{l=3}^mx_{2^{l-2}}^2+\sum_{t=0}^{m-4}2^{t}x_1^{2^{m-1-t}} \prod_{i=0}^{t}x_{2^{i+1}}^{2^{m-2-t}}\right).\]
\end{lemma}
\begin{proof}

Since all orbits of the  multiplicative action of the group $\I_{2^m}$ on  $\Zz_{2^m}$ are of the form 
\[\Omega_{2^m}^{2^l},\]
where $0\leq l \leq m$, it is important to find out behaviour of particular mappings  
\[x \rightarrow 3^{2^sr} x \pmod{n},\text{ where } x \in \mathbb{Z}_{2^m},\]
on those orbits. As stated before, we are interested in \textbf{ctypes} of those mappings on $\Omega_{2^m}^{2^l}$. In order to find \textbf{ctype} of mappings  $3^{2^sr}$ is restricted on orbit $\Omega_{2^m}^{2^l}$ we use results of Corollary \ref{rd-plus3-nam}.

Firstly, all elements from $\Gamma_1$ are fixing $\Omega^1_{2^m}=\{0\}$ and $\Omega^{2}_{2^m}=\{2^{m-1}\}$ as one-element orbits.

By considering parameter $l$ in the result of Corollary \ref{rd-plus3-nam}, we differentiate two major cases: when $l <3$ and when $l \geq 3$. Accordingly, the orbits
\[\Omega^{1}_{2^m}=\{0\},\,\,\Omega^{2}_{2^m}=\{2^{m-1}\}, \,\,\,\Omega^{2^2}_{2^m}=\{2^{m-2},2^m-2^{m-2}\}\]
should be treated separately. 

For example, for the mappings of the form $3^r$, where $r =1,3,\ldots, 2^{m-2-s}-1$, we have 
\[\ctype(3^r)=x_1^2x_2 \prod_{l=3}^{m}x_{2^{l-2}}^2.\]
The total number of these mappings from $\Gamma_1$ is $2^{m-3}$.

For those mappings from $\Gamma_1$  of the form $3^{2^sr}$, when $s \geq 1$, we have 
\[\ctype(3^{2^sr})=x_1^4 \prod_{l=3}^{s+2}x_{1}^{2^{l-1}}  \prod_{l=s+3}^{m}x_{2^{l-2-s}}^{2^{s+1}}=x_1^{2^{s+2}}\prod_{l=s+3}^{m}x_{2^{l-2-s}}^{2^{s+1}}.\]
The total number of these elements is $2^{m-3-s}$. clearly, for $r=0$, we get identity and ctype of it is $x_1^{2^m}$.

Finally, by adding all them together, after some elementary algebraic manipulation, we get a partial cycle index of the subset $\Gamma_1$

\[\mathcal{Z}_{(\Gamma_1,\Zz_{2^m})}(\I_{2^m})=\frac{1}{2^{m-1}}\left(x_1^{2^m}+ 2^{m-3}x_1^2x_2\prod_{l=3}^mx_{2^{l-2}}^2+\sum_{t=0}^{m-4}2^{t}x_1^{2^{m-1-t}} \prod_{i=0}^{t}x_{2^{i+1}}^{2^{m-2-t}}\right).\]

\end{proof}

\begin{lemma} \label{gamma2-parni}
Let us consider 
\[\Gamma_2=-\{3^{2^sr} \mid s=0,1,\ldots,m-3,\,r =0,1,3,\ldots, 2^{m-2-s}-1\},\]
as a subset of the group $\I_{2^m}$, for $m\geq 3$. Then, the partial cycle index of this set is
\[\mathcal{Z}_{(\Gamma_2,\Zz_{2^m})}(\I_{2^m})=\frac{1}{2^{m-1}}\left( x_1^2x_2^{2^m-2} +2^{m-3}x_1^4\prod_{l=3}^m x_{2^{l-2}}^2+x_1^2\sum_{t=0}^{m-4}2^{t}x_2^{2^{m-t-1}-1}
\prod_{i=1}^{t}x_{2^{i+1}}^{2^{m-t-2}}\right).\]
\end{lemma}

\begin{proof}
Similarly, as in the proof of the Lemma \ref{gama1-parni}, we use the result from Corollary \ref{rd-minus3-nam}. Note that
\[\ctype(-1)=x_1^2x_2^{2^m-2}.\]
By using same type of reasoning used to prove Lemma \ref{gama1-parni} we get the result.
\end{proof}

\begin{cor} \label{partition-u-2m}
Let $\Gamma_1$ and $\Gamma_2$ be subsets of the group $\I_{2^m}$, $m \geq 3$, as it has been introduced in the lemmas \ref{gama1-parni} and \ref{gamma2-parni}. Then 
\[\I_{2^m}=\Gamma_1 \uplus \Gamma_2.\]

\end{cor}

\begin{proof}
According to the Lemma \ref{generator-3}, every number from $\I_{2^m}$, if not $-1$ or $1$, has a form
\[(-1)^a3^{2^sr} \text{ where } a \in \{0,1\},\;s\in \{0,1,\ldots,m-3\} \text{ and r is an odd number},\]  
so $\Gamma_1 \cup \Gamma_2 = \I_{2^m}$. 
Also, from the same lemma it follows that $\Gamma_1 \cap \Gamma_2 =\emptyset$. Hence,  $\Gamma_1 \uplus \Gamma_2$ is disjoint union and equal to $\I_{2^m}$. 
\end{proof}

\begin{lemma} \label{zaparne-total} The cycle index 
$\mathcal{Z}=\mathcal{Z}_{(\I_{2^m},\Zz_{2^m})}$ of the permutation group $\I_{2^m}$ acting on $\Zz_{2^m}$ is
\begin{eqnarray}
 \nonumber \mathcal{Z}_{\,\,} &=& x_1^2\;\text{ if } m=1, \\
 \nonumber \mathcal{Z}_{\,\,} &=& \frac{1}{2}(x_1^4+x_1^2x_2)\;\text{ if } m=2, \\
  \nonumber \mathcal{Z}_{\,\,} &=& \frac{1}{2^{m-1}}(\mathcal{Z}_1+\mathcal{Z}_2)\text{ if } m\geq 3, \text{ where } \\
 \nonumber \mathcal{Z}_1 &= & x_1^{2^m}+ 2^{m-3}x_1^2x_2\prod_{l=3}^mx_{2^{l-2}}^2+\sum_{t=0}^{m-4}2^{t}x_1^{2^{m-1-t}} \prod_{i=0}^{t}x_{2^{i+1}}^{2^{m-2-t}},\\
 \nonumber \mathcal{Z}_2 &=&  x_1^2x_2^{2^m-2} +2^{m-3}x_1^4\prod_{l=3}^m x_{2^{l-2}}^2+x_1^2\sum_{t=0}^{m-4}2^{t}x_2^{2^{m-t-1}-1}
\prod_{i=1}^{t}x_{2^{i+1}}^{2^{m-t-2}}.
\end{eqnarray}

\end{lemma}

\begin{proof}

The claim follows trivially for $m=1$ and $m=2$. For the case $m\geq 3$, we use results given in Lemma \ref{gama1-parni} and Lemma \ref{gamma2-parni}, combined with the fact given in Corollary \ref{partition-u-2m}, that
\[\I_{2^m}=\Gamma_1 \uplus \Gamma_2.\]
\end{proof}

\subsection{Cycle index od $\I_{p^m}$ - $p$ an odd, prime number}

The following lemma considers the case of $\I_{p^m}$, where $p$ is an od prime number. As it has been already noted, this group is cyclic and consequently, it is much easier to find its cycle index.
\begin{lemma} \label{zaneparne-total} Let $p$ be an odd prime. The cycle type of the permutation group $\I_{p^m}$ acting on $\Zz_{p^m}$ is
\[\mathcal{Z}_{(\I_{p^m},\Zz_{p^m})}=\frac{1}{\phi(p^m)}\sum_{k=1}^{\phi\left(p^m\right)}\prod_{i=0}^{m} x_{u(i,k)}^{v(i,k)},\]
where \[v(i,k)=(\phi(p^i),k) \text{ and } u(i,k)=\frac{\phi(p^i)}{v(i,k)}.\]
\end{lemma}
\begin{proof} It is well known that in the case of odd prime $p$, the automorphism group $\I_{p^m}$ is cyclic \cite{Zass}. Let $\beta$ be a generator of $\I_{p^m}$. Then, $\ord(\beta)=\phi(p^m)$. It is elementary fact that in an arbitrary group $G$ and $g \in G$, such that $\ord(g)=n$ it holds that 
\[\ord(g^k)=\frac{n}{(k,n)}.\] 
Since $r_{\beta}(p^i)=\phi(p^i)$, for  $i=0,1,\ldots,m$, we conclude that
\[r_{\beta^k}(p^i)=\frac{\phi\left(p^i\right)}{(k,\phi(p^i))}, \text{ for } i=0,1,\ldots,m.\]
Now, the claim follows directly from the result (\ref{cycle-index-compact}) given in the Corollary \ref{nice-compact}.
\end{proof}
\subsection{Cycle index of direct product of permutation groups}
Since we found the cycle indices of all groups $\I_{p^m}$ when $p$ is a prime number, there is a natural question if there exists a way to combine them together in order to obtain the cycle index of $\I_n$, where $n$ is the product of those prime power components. Hence, we need something like the cycle index of the direct product of permutation groups.

Let $G_1,G_2$ be permutation groups acting on sets $X_1,X_2$ respectively. Let $G=G_1 \times G_2$ and $X=X_1\times X_2$ be the direct product of corresponding groups and sets. For an element $x=(x_1,x_2)$ of $X$ and an element $g=(g_1,g_2)$ of $G$, we define the action of $g$ on $x$ by
\[(g,a)\mapsto (g_1x_1,g_2x_2)\]
Evidently, $G$ is a permutation group on $X$.
Let $P$ and $Q$ be polynomials
\[P(x_1,x_2,\ldots,x_u)=\sum a_{i_1i_2\ldots i_u}x_1^{i_1}x_2^{i_2}\ldots x_u^{i_u},\]
\[Q(x_1,x_2,\ldots,x_v)=\sum b_{j_1j_2\ldots j_v}x_1^{j_1}x_2^{j_2}\ldots x_v^{j_v}\]
In \cite{HarHigh} the following product operator was defined
\[P\circledast Q=\sum a_{i_1i_2\ldots i_u} b_{j_1j_2\ldots j_v} \prod_{\substack{1\leq l \leq u \\ 1\leq m \leq v}}(x_l^{i_1} \circledast x_m^{j_m}),\]
where
\[x_l^{i_1} \circledast x_m^{j_m}=x_{\lcm(l,m)}^{i_1j_m\gcd(l,m)}\]
We need the following lemma. For proof, see \cite{HarHigh} and \cite{Wei}.
\begin{lemma} \label{prodpermgroup} The cycle index of the natural action of permutation group $G_1 \times G_2$ on $X_1 \times X_2$ induced by
actions $G_1$ on $X_1$ and $G_2$ on $X_2$ can be expressed as:
\[\mathcal{Z}_{(G_1\times G_2,X_1 \times X_2)}=\mathcal{Z}_{(G_1,X_1)} \circledast \mathcal{Z}_{(G_2,X_2)}.\]
\end{lemma}
Let $n=\prod_{i=1}^sp_i^{\alpha_i}$. Applying the ring isomorphism  \[\Zz_n \cong \bigoplus_{i=1}^s\Zz_{p_i^{\alpha_i}},\] it follows that
\[\I_n\cong\bigoplus_{i=1}^s\I_{p_i^{\alpha_i}}.\]
Hence, according to Lemma \ref{prodpermgroup}, we have
\[\mathcal{Z}_{(\I_n,\Zz_n)}=\mathcal{Z}_{(\I_{p_1^{\alpha_1}},\Zz_{p_1^{\alpha_1}})}\circledast \mathcal{Z}_{(\I_{p_2^{\alpha_2}},\Zz_{p_2^{\alpha_2}})} \circledast \cdots \circledast \mathcal{Z}_{(\I_{p_s^{\alpha_s}},\Zz_{p_s^{\alpha_s}})}.\]
Since cycle indices of prime power components are given in Lemma \ref{zaparne-total} and Lemma \ref{zaneparne-total}, the cycle index $\mathcal{Z}_{(\I_n,\Zz_n)}$ can be calculated as above. 
\begin{example}
Let us find the cycle index of $\I_{60}$. Frome lemmas \ref{zaparne-total} and \ref{zaneparne-total} we know that
\[\mathcal{Z}_{(\I_{2^2},\Zz_{2^2})}=\frac{1}{2}(x_1^4+x_1^2x_2), \mathcal{Z}_{(\I_{3},\Zz_{3})}=\frac{1}{2}(x_1^3+x_1x_2), \mathcal{Z}_{(\I_{5},\Zz_{5})}=\frac{1}{4}(x_1^5+2x_1x_4+x_1x_2^2).\]

Therefore, and according to Lemma \ref{prodpermgroup}
\[\mathcal{Z}_{(\I_{60},\Zz_{60})}=\mathcal{Z}_{(\I_{2^2},\Zz_{2^2})}\circledast \mathcal{Z}_{(\I_{3},\Zz_{3})} \circledast  \mathcal{Z}_{(\I_{5},\Zz_{5})}.\]

\noindent Firstly, we calculate product of the first two cycle indices, that actually is $\mathcal{Z}_{(\I_{12},\Zz_{12})}$. So,
\begin{dmath*}
\mathcal{Z}_{(\I_{12},\Zz_{12})}=\mathcal{Z}_{(\I_{2^2},\Zz_{2^2})}\circledast \mathcal{Z}_{(\I_{3},\Zz_{3})}
=\frac{1}{2}(x_1^4+x_1^2x_2)\circledast \frac{1}{2}(x_1^3+x_1x_2)=
\frac{1}{4}(x_1^{12}+x_1^4x_2^4+x_1^2x_2^5+x_1^6x_2^3).
\end{dmath*}

Finally, we get 

\begin{dmath*}
\mathcal{Z}_{(\I_{60},\Zz_{60})}=\mathcal{Z}_{(\I_{12},\Zz_{12})}\circledast \mathcal{Z}_{(\I_{5},\Zz_{5})}=
\frac{1}{4}(x_1^{12}+x_1^4x_2^4+x_1^2x_2^5+x_1^6x_2^3)\circledast \frac{1}{4}(x_1^5+2x_1x_4+x_1x_2^2)=
\frac{1}{16}(x_1^{60}+2x_1^4x_2^4x_4^{12}+x_1^4x_2^{28}+2x_1^2x_2^5x_4^{12}+
x_1^2x_2^{29}+2x_1^6x_2^3x_4^{12}+
x_1^6x_2^{27}+2x_1^{12}x_4^{12}+x_1^{12}x_2^{24}+x_1^{10}x_2^{25}+
x_1^{20}x_2^{20}+x_1^{30}x_2^{15}).
\end{dmath*}
Certainly, the same result would be obtained by simple application of the formula (\ref{cycle-index-compact}) in the Corollary \ref{nice-compact}.

\end{example}

\section{Conclusions}
In this paper we studied the group action of the automorphism group $\I_n=\aut(\Zz_n)$ on the set $\Zz_n$, that is the set of residue classes modulo $n$. The main goal of the paper was to find the cycle index of that action. Based on some elementary number theory and algebraic techniques, we get nice, compact result in the Corollary \ref{nice-compact}. Also, in the Lemma \ref{zaparne-total} and Lemma \ref{zaneparne-total}, we provided technically more detailed look at the building blocks of cycle index of the studied group action.

In the further research, as we announced in the Introduction, it could be interesting to examine some combinatorial problems as a problem of finding the number of orbits or equivalence classes of subsets of $\Zz_n$. Namely, there is a natural way to induce the discussed group action on the set $\Ov_k$, standing for the set of all subsets of $\Zz_n$ of size $k \leq n$ and then the task could be principally resolved by P\'olya's theory application as in \cite{Polya,Bruij,SurPol}.

It is worth of mentioning that the number of orbits of sets of $\Ov_k$ is related to the problem of factorizations of abelian groups into direct product of subsets that is examined in \cite{Szabo, Sands, Boz}. Hence, it seems that some further research in this topic could be very fruitful.

\bibliographystyle{amsplain}

\end{document}